\documentclass[11pt,reqno]{amsart}

\usepackage{pb-diagram, graphics}
\usepackage{setspace}
\usepackage{amscd,amsfonts}
\usepackage{amssymb}
\usepackage{amsmath}
\usepackage{amsthm}
\theoremstyle{plain}
\newtheorem{thm}{Theorem}[subsection]
\newtheorem{cor}[thm]{Corollary}
\newtheorem{lem}[thm]{Lemma}
\newtheorem{prop}[thm]{Proposition}

\theoremstyle{definition}
\newtheorem{defn}[thm]{Definition}

\newtheorem{rem}[thm]{Remark}
\newtheorem{rems}[thm]{Remarks}

\newcommand{\B}{\mathcal{B}}

\newcommand{\dive}{\operatorname{div}}

\newcommand{\0}{\bar 0}
\newcommand{\1}{\bar 1}

\numberwithin{equation}{subsection}

\def\Z{{\mathbb Z}}

\def\al{\alpha}

\def\N{\mathbb N}
\def\bbc{\mathbb C}
\def\:{\colon}
\def\a{\alpha}

\newcommand{\fg}{\mathfrak{g}}

\def \fg{\mathfrak{g}}

\def \fh{\mathfrak{h}}

\def \fh{\mathfrak{h}}

\def\La{\mathfrak{g}}

\def\Lsa{\mathfrak{h}}
\def\C{{\mathbb C}}

\def\Z{{\mathbb Z}}
\def\bbz{{\mathbb Z}}

\def\N{{\mathbb N}}

\def\bu{\textbf{U}}
\def\f{\mathcal{F}}
\newcommand{\supp}{\operatorname{supp}}
\newcommand{\lie}[1]{\mathfrak{#1}}\def\span{\textnormal{span}}
\def\m{\mathcal{M}}
\def\bb{\textbf{B}}

\def\cs{\mathcal{S}}
\def\cp{\mathcal{P}}

\def\bu{\textbf{U}}

\begin{document}
\normalsize

\title[Integral Bases for the Universal Enveloping Superalgebras of Map Superalgebras II]
{Integral Bases for the Universal Enveloping Superalgebras of Map Superalgebras II}

\author{Irfan Bagci}
\address{Department of Mathematics \\
University of North Georgia \\
Oakwood, GA 30566}
\email{bagcimath@gmail.com}

\author{Samuel Chamberlin}
\address{Computer Science and Mathematics Department\\
Park University\\
Parkville, MO 64152.}
\email{samuel.chamberlin@park.edu}

\begin{abstract}
Let $\fg$ be a finite dimensional complex simple  Lie superalgebra of Cartan type  and $A$ be a commutative, associative algebra with unity over $\bbc$. In this paper, following \cite{BC},  we define an integral form for the universal enveloping superalgebra of the map superalgebra $\fg\otimes A$, and exhibit an explicit integral basis for this integral form.
\end{abstract}
\maketitle

\section{Introduction}

Integral bases for various classes of Lie algebras and universal enveloping algebras were formulated by Chevalley, Cartier, Kostant, Garland and Mitzman, \cite{Kos, Gar, Mit}.
Integral bases for quantized universal enveloping algebras of various Lie algebras were discovered by Lusztig, Beck, Chari and Pressley, \cite{Lus, Bec, Cha}.
%Some of these integral bases were then used to study the representation theory of these algebras in characteristic $p$.
Integral bases for the universal enveloping algebras map algebras of the simple complex Lie algebras were given in \cite{C}.

In 1977, V. Kac provided a complete classification of simple Lie superalgebras, \cite{Kac}. The simple finite-dimensional Lie superalgebras are divided into two types based on their even part: they are either classical or of Cartan type. Integral bases for the universal enveloping algebras map algebras of the classical Lie superalgebras were given in \cite{BC}. Integral bases for the universal enveloping algebras of Cartan type Lie superalgebras were given in \cite{Gav}.

In this work we finish the construction of integral bases for the universal enveloping algebras of map superalgebras started in \cite{BC} by formulating these bases for map superalgebras of Cartan type Lie superalgebras. This is done via straightening identities in the universal enveloping superalgebras. The root systems of Cartan type Lie superalgebras have properties which make them quite different from those of the simple classical Lie superalgebras. In particular, not every root has multiplicity one and that the negative of a root is not always a root. These differences require a more complicated Chevalley type basis. Due to the difference in the Chevalley type basis, the straightening identities, which are required to prove the existence of a suitable integral basis, are more difficult in the Cartan type Lie superalgebra setting.

This paper is organized as follows:  In Section 2 we fix some notation and briefly review basic facts about simple Lie superalgebras of Cartan type, introduce map superalgebras and record the properties we are going to need in the rest of the paper. Then in Section 3 we state the main theorem of the paper and give some important corollaries. Next in Section 4 we state and prove all of the necessary straightening identities. In Section 5 we prove the main result of the paper and give a triangular decomposition of our integral form.

\section{ Preliminaries}

\subsection{}\label{defdeg}

The following notation will be used throughout this manuscript: $\C$ is the set of complex numbers, $\Z_{\geq0}$ is the set of non-negative integers, and $\N$ is the set of positive integers. All vector spaces and algebras we consider will be over the ground field $\C$. A Lie superalgebra is a finite dimensional $\Z_2$-graded vector space $\fg=\fg_{\0}\oplus\fg_{{\1}}$ with a bracket $[,]:\fg\otimes\fg\rightarrow\fg$ which preserves the $\Z_2$-grading and satisfies graded versions of the operations used to define Lie algebras. The even part $\fg_{\0}$ is a Lie algebra under the bracket.

Given any Lie superalgebra $\fg$ let $\bu(\fg)$ be the universal enveloping superalgebra of $\fg$.  $\bu(\fg)$ admits a PBW type basis and if   $x_1, \cdots, x_m$ is  a basis of $\fg_{\0}$ and $y_1, \dots, y_n$ is a basis of $\fg_{\1}$, then the elements
$$x_1^{i_1} \dots x_m^{i_m}y_1^{j_1}\dots y_n^{j_n}\ \  \text{with} \ \ i_1, \dots, i_m \geq 0  \ \ \text{and} \ \ j_1, \dots, j_n \in \{0, 1\}$$
form a basis of the universal enveloping superalgebra $\bu(\fg)$. Given $u\in\bu(\fg)$ and $r\in\Z_{\geq0}$ define
$$u^{(r)}:=\frac{u^r}{r!}\textnormal{ and }\binom{u}{r}:=\frac{u(u-1)\cdots(u-r+1)}{r!}.$$
Define $T^0(\fg):=\C$, and for all $j\geq1$, define $T^j(\fg):=\fg^{\otimes j}$, $T(\fg):=\bigoplus_{j=0}^\infty T^j(\fg)$, and $T_j(\fg):=\bigoplus_{k=0}^jT^k(\fg)$. Then set $\bu_j(\fg)$ to be the image of $T_j(\fg))$ under the canonical surjection $T(\fg)\to\bu(\fg)$. Then for any $u\in\bu(\fg)$ \emph{define the degree of $u$} by $$\deg u:=\min_{j}\{u\in\bu_j(\fg)\}$$

\subsection{}

%A simple Lie superalgebra $\fg=\fg_0\oplus\fg_1$ is called \emph{classical} if its even part $\fg_0$ is a reductive Lie algebra.
Finite-dimensional complex simple Lie superalgebras were classied by Kac \cite{Kac}.  The finite-dimensional simple Lie superalgebras are divided into two classes: the classical (when the even part is a reductive Lie algebra)
and the Cartan ones (otherwise). Lie superalgebras of Cartan type consists of four infinite families $$W(n), \ n\geq 2; \ S(n), \ n\geq 3; \  \tilde{S}(n), \ n\geq 4  \  \text{ and even}; \  H(n), \ n \geq 4.$$

 We briefly define Cartan type Lie superalgebras. Assume that $n\geq 2$.  Let $\Lambda(n)$ denote the exterior algebra on $n$ odd generators $\xi_1, \dots, \xi_n$; $\Lambda(n)$ is a $2^n$-dimensional associative algebra we assign to it a $\Z$-grading by setting deg  $\xi_i = 1$ for $1\leq i\leq n$. The $\Z_2$-grading is inherited from the $\Z$-grading by setting $\Lambda(n)_{\0} = \bigoplus_{k} \Lambda^{2k}(n)$ and $\Lambda(n)_{\1} = \bigoplus_{k} \Lambda^{2k+1}(n)$.

Let $p(x)$ denote the parity of a homogeneous  element $x$ in a $\Z_2$-graded vector space. A (homogeneous) \emph{superderivation} of $\Lambda(n)$ is a linear map $D: \Lambda(n) \to \Lambda(n)$ which satisfies $D(xy)=D(x)y + (-1)^{p( D) \; p(x)}xD(y)$ for all homogenous $x,y \in \Lambda(n)$. % where $\deg$ denotes the degree of a homogeneous element in %$\Z_2$-grading.
 Then $W(n)$ is the Lie superalgebra of super derivations of $\Lambda(n)$. The bracket on  $W(n)$ is the supercommutator bracket. That is,
$$[x,y]=xy-(-1)^{ p(x) p( y)}yx$$
for homogeneous $x,y$ and extended to all $W(n)$ by linearity.  The $\Z$-grading on $\Lambda(n)$ induces a $\Z$-grading on $W(n)$
$$W(n)= W(n)_{-1} \oplus W(n)_0 \oplus \dots \oplus W(n)_{n-1},$$
where $W(n)_k$ consists of derivations that increase the degree of a homogeneous element by $k$. %The $\Z_2$-grading on $W(n)$ is inherited from the $\Z$-grading by setting $W(n)_{\0}= %\oplus_{k} W(n)_{2k}$ and $W(n)_{\1}= \oplus_{k} W(n)_{2k+1}.$
 %Since the bracket preserves  the $\Z$-grading $W(n)_0$ is a Lie subalgebra of $W(n)$ and  %$W(n)_0 \cong \mathfrak{gl}(n).$

Denote by  $\partial_i$ the derivation of $\Lambda(n)$ defined by
$$\partial_i(\xi_j): = \delta_{i,j}.$$ Then any element $D$ of $W(n)$ can be written in the form
$$\sum_{i=1}^nf_i\partial_i,$$
where $f_i \in \Lambda(n)$.

The superalgebra $S(n)$ is the subalgebra of $W(n)$ consisting of all elements $D \in W(n)$ such that
$\dive(D) = 0$, where
$$\dive\left(\sum_{i=1}^n f_i \partial_i\right): = \sum_{i=1}^n \partial_i(f_i).$$
The superalgebra $S(n)$ has a $\Z$-grading induced by the grading of $W(n)$
$$S(n)= S(n)_{-1} \oplus S(n)_0 \oplus \dots \oplus S(n)_{n-2}.$$
%and $S(n)_0$ is isomorphic to $\sl(n)$.

The simple Lie superalgebra $\tilde{S}(n)$ is defined only when $n$ is even. The superalgebra $\tilde{S}(n)$ has a $\Z$-grading
$$\tilde{S}(n)= \tilde{S}(n)_{-1} \oplus \tilde{S}(n)_0 \oplus \dots \oplus \tilde{S}(n)_{n-2}.$$
For each $r$ with $0\leq r \leq n-2, \tilde{S}_r(n)= S_r(n)$ . The difference is that $\tilde{S}(n)_{-1}$ has a basis consisting of $\xi_1 \dots \xi_n \partial_i$.
This grading is not an algebra grading.

The subspace of $W(n)$ spanned by all super derivations of the form
$$D_f:=\Sigma _{i\in \N}\partial _i(f)\partial _i$$
where $f\in \Lambda(n)$, is a Lie superalgebra called $\tilde{H}(n)$. It inherits a natural $\Z$-grading from $W(n)$ and we have
$$\tilde{H}(n)=\tilde{H}(n)_{-1} \oplus \tilde{H}(n)_0 \oplus \dots \oplus \tilde{H}(n)_{n-2} .$$
The subalgebra $$H(n)=[\tilde{H}(n), \tilde{H}(n)]=H(n)_{-1} \oplus H(n)_0 \oplus \dots \oplus H(n)_{n-3}$$ is a simple Lie superalgebra of Cartan type.
%$H(n)_0 \cong \mathfrak{so}(n)$  as a Lie algebra and the homogeneous component $H(n)_r$ %is isomorphic as an $H(n)_0$-module to $\Lambda(n)_{r+2}$ via
%$D_f\mapsto f$. %Thus the superderivations $x_I=D_{\xi _I}$, where $\empty \neq I\subseteq %N$, form a basis for $\tilde{H}(n)$, and
%$\dim H(n)_r=\binom{n}{r+2}$.

Throughout this work $\fg$ will be a Cartan type Lie superalgebra. The crucial difference between the Cartan type superalgebras and the classical superalgebras
is that the $\fg_{\0}$ component is no longer reductive.  However, as was described above, the Cartan type Lie superalgebras  admit a ${\mathbb Z}$-grading: $\fg=\bigoplus_{k\in {\mathbb Z}} \fg_{k}.$  The grading is compatible with the $\Z_{2}$-grading in the sense that $\bigoplus_{k}\fg_{2k}=\fg_{\0}$ and $\bigoplus_{k}\fg_{2k+1}=\fg_{\1}$.
%Since  the bracket respects the grading (i.e., \ $[\fg_{i},\fg_{j}] \subseteq \fg_{i+j}$ for all %integers $i,j$)  $\fg_0$ is a Lie subalgebra of $\fg$ .  $\fg_0 \cong \mathfrak{gl}(n)$ if $\fg= %W(n)$,  $\fg_0 \cong \mathfrak{sl}(n)$ if $\fg= S(n)$ or $\tilde{S}(n)$, and $\fg_0 \cong %\mathfrak{so}(n)$ if $\fg= H(n)$

\subsection{Cartan Subalgebras and the Root Structure}
Cartan subalgebras $\fh$ of $\fg$ coincide with Cartan subalgebras of $\fg_{0}$.  We fix a distinguished Cartan subalgebra $\fh$ of $\fg_{0}$ so that $\fh$ has the following bases respectively.
\begin{eqnarray*}
\left\{\xi_k\partial_k\ |\ k\in\{1,\ldots,n\}\right\}&\textnormal{ if }&\fg=W(n) \\
\left\{\xi_k\partial_k-\xi_{k+1}\partial_{k+1}\ |\ k\in\left\{1,\ldots,n-1\right\}\right\}&\textnormal{ if }&\fg=S(n)\textnormal{ or }\tilde{S}(n)\\
\left\{\xi_k\partial_k-\xi_{\left[\frac{n}{2}\right]+k}\partial_{\left[\frac{n}{2}\right]+k}\ |\ k\in\left\{1,\ldots,\left[\frac{n}{2}\right]\right\}\right\}&\textnormal{ if }&\fg=H(n)
\end{eqnarray*}

If $\fg=W(n)$ define $\mathcal{E} \in\fg_0$ by
$$\mathcal{E}:=\sum_{i=1}^n\xi_i\partial_i$$
Then define $\overline{\fh}$, $\overline{\fg}_0$, and $\overline{\fg}$ as follows:
\begin{eqnarray*}
\overline{\fh}&:=&\fh\textnormal{ if } \fg=W(n)\textnormal{ or }\tilde{S}(n)  \ \ \textnormal{and}   \ \  \ \ \overline{\fh}:=\fh+\C\mathcal{E}\textnormal{ if } \fg=S(n)\textnormal{ or }H(n)\\
\overline{\fg}_0&:=&\fg_0\textnormal{ if } \fg=W(n)\textnormal{ or }\tilde{S}(n) \ \ \textnormal{and}   \ \ \overline{\fg}_0:=\fg_0+\C\mathcal{E}\textnormal{ if } \fg=S(n)\textnormal{ or }H(n)\\
\overline{\fg}&:=&\fg\textnormal{ if } \fg=W(n)\textnormal{ or }\tilde{S}(n)  \ \ \textnormal{and}   \ \ \overline{\fg}:=\fg+\C\mathcal{E}\textnormal{ if } \fg=S(n)\textnormal{ or }H(n)
%\overline{\fh}&:=&\fh+\C\mathcal{E}\textnormal{ if } \fg=S(n)\textnormal{ or }H(n)\\
%\overline{\fg}_0&:=&\fg_0+\C\mathcal{E}\textnormal{ if } \fg=S(n)\textnormal{ or }H(n)\\
%\overline{\fg}&:=&\fg+\C\mathcal{E}\textnormal{ if } \fg=S(n)\textnormal{ or }H(n)
\end{eqnarray*}

Then we have a root decomposition
$$\fg=\bigoplus_{\alpha \in \overline{\fh}^{\ast}}\fg_{\alpha},$$
where $\fg_{\alpha}:=\{x\in\fg\mid [h,x] = \alpha(h)x\ \text{for all }\ h \in \overline{ \fh}\}$. The set $R:=\{\alpha\in \overline{\fh}^\ast-\{0\} \mid \fg_{\alpha} \neq \{0\}\}$ is called the root system. The $\Z_2$-grading of $\fg$ determines a decomposition of $R$ into the disjoint union of the even roots $R_{\0}$  and the odd roots $R_{\1}$.

For each $\a\in R$ there is a unique integer $ht(\a)$ (called the height of $\a$) such that
$\La\subset\fg_{ht(\a)}$. This allows us to use the $\Z$ grading of $\fg$ to determine a disjoint decomposition of $R$. For $z\in\Z$ we define $R_z:=\{\a\in R\ |\  ht(\a)=z\}$. Then
$$R=\bigcup_{z\in\Z}R_z$$
The $\Z_2$ and $\Z$ decompositions are compatible in that
$$R_{\0}=\bigcup_{z\in\Z}R_{2z}\ \textnormal{ and }\ R_{\1}=\bigcup_{z\in\Z}R_{2z+1}$$

%Define $\tilde{R}:=\{(\a,j)\ |\ \a\in R,\ j\in\{1,\ldots,\mu(\a)\}$, $\tilde{\pi}:\tilde{R}\to R$ by $\tilde{\pi}(\a,j):=\a$, and $\tilde{R}_z:=\tilde{\pi}^{-1}(R_z)$ for all $z\in\{\0,\1,-1,0,\ldots,n\}$.

%Recall from the previous section that we fixed a maximal torus $\fh \subseteq \fg_{0} \subseteq \fg.$  With respect to this choice we have a root decomposition
%$$\La=\Lsa\oplus \bigoplus_{\alpha \in\Phi }\La_{\alpha}.$$

Let us describe the roots. If $\fg = W(n)$ then  $\fg_0 \cong \mathfrak{gl}(n)$. We choose the standard basis $\varepsilon _1,\dotsc , \varepsilon _n$ of $\Lsa^\ast$ where $\varepsilon_{i}(\xi_{j}\partial_{j})=\delta_{i,j}$ for all $1 \leq i,j \leq n.$  Then the root system of $\fg$ is the set
$$R=\{\varepsilon _{i_1}+ \dotsb  +\varepsilon _{i_k}-\varepsilon _j\mid 1\leq i_1<\dotsb <i_k\leq n,\ 1\leq j\leq n\}.$$

If $\fg = S(n)$ then  $\fg_0 \cong \mathfrak{sl}(n)$. The root system of $S(n)$ is a subset of that of $W(n)$. The root system of $S(n)$ is obtained from the root system of $W(n)$ by removing the roots $\varepsilon _{1}+ \dotsb  +\varepsilon _{n}-\varepsilon _j $. Thus the  root system of $\fg$ is the set
$$R=\{\varepsilon _{i_1}+ \dotsb  +\varepsilon _{i_k}-\varepsilon _j\mid 1\leq i_1<\dotsb <i_k\leq n,\ k<n, \ \ 1\leq j\leq n\}.$$

If $\fg = \tilde{S}(n)$ then  $\fg_0 \cong \mathfrak{sl}(n)$ and the root system of $\fg$ is the same as that of $S(n)$ except in this case we have $\varepsilon _1+\dotsc+\varepsilon_n=0.$

Finally if  $\fg = H(n)$ then  $\fg_0 \cong \mathfrak{so}(n)$. Let $\{\varepsilon _1,\dotsc , \varepsilon _{\left[\frac{n}{2}\right]}\}$ be the standard basis in the weight space of $\mathfrak{so}(n)$ and $\delta$ be dual to $\mathcal{E}$. If $n$ is even, then the  root system of $\fg$ is the set
$$R=\{\pm \varepsilon _{i_1}\pm \dotsb  \pm \varepsilon _{i_k}+ m\delta \mid 1\leq i_1<\dotsb <i_k\leq n/2,\ k-2 \leq m \leq n-2, \ k-m \in 2\mathbb{Z}\}.$$
If $n$ is odd, then the  root system of $\fg$ is the set
$$R=\{\pm \varepsilon _{i_1}\pm \dotsb  \pm \varepsilon _{i_k}+ m\delta \mid 1\leq i_1<\dotsb <i_k\leq n/2,\ k-2 \leq m \leq n-2\}.$$

Many properties of root decompositions for semisimple Lie algebras do not hold in our case. For example, a root may have multiplicity bigger than one. Also, $\a \in R$ does not imply that $-\a\in R$. In the following proposition we record some properties of the roots; see also \cite{Gav,Kac,Sch,Ser}.

\begin{prop}\label{rootprop} \cite{Gav, Ser}
Let $\fg$ be a Cartan type  Lie superalgebra and let  $\fg= \oplus_{\alpha \in \overline{\fh}^{\ast}}\fg_{\alpha}$ be its root decomposition relative to $\fh$.
\begin{itemize}
\item[(1)]If  both $\alpha \in R$ and  $-\alpha \in R$, then $\fg_{\al}, \fg_{-\al}\subset \fg_0$ and $\dim \fg_{\alpha} =1$.
\item[(2)] $2 \alpha \in R$ if $\fg= H(n), \  \alpha \in \{m \delta \mid m \in \Z\}$ or $\fg =  \tilde{S}(n), \ \alpha \in  \{-\varepsilon_i \mid i = 1, \dots, n\}$.
%\item[(3)] $3 \alpha \in R$ if $\fg= H(n), \  \alpha \in \{m \delta \mid m \in \Z\}$
\end{itemize}
\end{prop}

%\subsection{} We use the following notational conventions throughout. $\fg$  will be Cartan type Lie superalgebra. Let $\fg^{+}=\fg_1\oplus \dots \oplus \fg_{n-1}$ and  $\fg^{-}=\La_{-1}$, so that $\fg$ has the lopsided triangular decomposition $$\fg=\fg^{-}\oplus \fg_{0} \oplus \fg^{+}.$$

As in \cite{Gav} and \cite{Ser} fix a distinguished simple root system for $\fg$, $\Delta:=\{\alpha_1,\ldots,\alpha_l\}$. Denote by  $R^+$, and $R^-$,  positive roots, negative roots respectively. Define $\fg^\pm:=\displaystyle\bigoplus_{\a\in R^\pm}\fg_{\a}$. Then we have the lopsided triangular decomposition $\fg=\fg^-\oplus\fh\oplus\fg^+$. For $j\in\{\0,\1\}$ define $R_j^\pm:=R^\pm\cap R_j$. Define $I:=\{1,\ldots,l\}$.

\subsection{Chevalley Bases}

In this subsection we define the notion of Chevalley bases for the Cartan type Lie superalgebras. The basis given here is identical to that in \cite{Gav}.

Define the multiplicity function $\mu:R\to\Z_{\geq0}$ by $\mu(\alpha):=\dim\fg_\alpha$ and for each $z\in\Z_{\geq0}$ define
$$[z]:=\left\{\begin{array}{cc}
\{1,\ldots,z\}\textnormal{ if }z\geq1\\
\varnothing\textnormal{ if }z=0
\end{array}\right.$$

\begin{defn}
A \emph{Chevalley basis} of $\fg$ is any $\C$-basis $B=\{h_i\ |\ i\in I\}\cup\{x_{\alpha,k}\ |\ \alpha\in R,\ k\in[\mu(\alpha)]\}$ for $\fg$ which is homogeneous for the cyclic grading of $\fg$ and has the following properties:
\begin{enumerate}
\item $\{h_i\ |\ i\in I\}$ is a $\C$-basis for $\fh$, such that $\alpha(h_i)\in\Z$ for all $\a\in R$ and all $i\in I$, $h_{\beta}\in%\fh_{\Z}:=
\span_{\Z}\{h_i\ |\ i\in I\}$ for all $\beta\in R_0$, and $[h_i,h_j]=0$ for all $i,j\in I$;
\item For all $\a\in R$, $\{x_{\a,k}\ |\ k\in[\mu(\alpha)]\}$ is a $\C$-basis for $\fg_{\a}$, and $[h_i,x_{\a,k}]=\a(h_i)x_{\a,k}$ for all $i\in I$, $\a\in R$, and $k\in[\mu(\alpha)]$;
\item $[x_{\a,k},x_{\a,k}]=0$ for all $\a\in R_{\0}$ and $k\in[\mu(\alpha)]$;
\item $[x_{\a,1},x_{-\a,1}]=h_{\a}$ (as defined in \cite{Ser}) for all $\a\in R_0$, and, for each $\gamma\in R_{-1}$, there exist injections $\sigma_\gamma:[\mu(-\gamma)]\hookrightarrow I$ such that $\left\{\pm h_{\sigma_{\gamma(k)}}\ |\ \gamma\in R_{-1},\ k\in[\mu(-\gamma)]\right\}=\{\pm h_i\ |\ i\in I\}$ and $[x_{\a,1},x_{-\gamma,k}]=\pm h_{\sigma_{\gamma(k)}}$;
\item $[x_{\a,k},x_{\beta,m}]=0$ for all $\a,\beta\in R$, $k\in[\mu(\a)]$, $m\in[\mu(\beta)]$ with $\a+\beta\notin R\cup\{0\}$;
\item Given $\a,\beta\in R$ with $\a+\beta\in R$, $k\in[\mu(\a)]$, and $m\in[\mu(\beta)]$ there exist $c_{\a,k}^{\beta,m}:[\mu(\a+\beta)]\to\{0,\pm1,\pm2\}$ such that $\displaystyle\left[x_{\a,k},x_{\beta,m}\right]=\sum_{j=1}^{\mu(\a+\beta)}c_{\a,k}^{\beta,m}(j)x_{\a+\beta,j}$, and the following hold:
\begin{enumerate}
\item If $\a,\beta\in R_0$, then $c_{\a,1}^{\beta,1}(1)=\pm(r+1)$ where $r\in \N$ is such that $\beta-r\a\in R$ and $\beta-(r+1)\a\notin R$ (note that in this case $\mu(\a)=\mu(\beta)=\mu(\a+\beta)=1$),
\item If $\a$ or $\beta$ is in $R_{-1}$, then there is one and only one $j^\prime\in[\mu(\a+\beta)]$ such that $c_{\a,k}^{\beta,m}(j^\prime)=\pm1$ and for all other $j\in[\mu(\a+\beta)]$ $c_{\a,k}^{\beta,m}(j)=0$,
\item If $\a=\beta$, there is at most one $j^\prime\in[\mu(\a+\beta)]$ such that $c_{\a,k}^{\beta,m}(j^\prime)\neq0$, and if such a $j^\prime$ exists then $c_{\a,k}^{\beta,m}(j^\prime)=\pm2$;
\end{enumerate}
\item If $\a\in R_0$ and $2\a+\beta\in R$, there exists a unique $k\in[\mu(2\a+\beta)]$ such that $[x_{\a,1},[x_{\a,1},x_{\beta,m}]]=\pm2x_{2\a+\beta,k}$ for all $m\in[\mu(\beta)]$
\end{enumerate}
\end{defn}

Fix a Chevalley basis for $\fg$,
$$B:=\{h_i\ | \ i\in I\}\cup\{x_{\a,k}\ |\ \a\in R,\ k\in[\mu(\a)]\}$$

\subsection{Map superalgebra and Notation}

Fix a commutative associative unitary algebra $A$ over $\C$. $\bb$ will denote a fixed $\C$--basis of $A$. The \emph{map superalgebra} of $\lie{g}$ is the $\Z_2$-graded vector space  $\lie{g}\otimes A$, where   $(\lie{g}\otimes A)_0 := \lie{g}_0\otimes A $ and  $(\lie{g}\otimes A)_1:= \lie{g}_1\otimes A $, with bracket given by linearly extending the bracket
$$[z\otimes a, z'\otimes b]:=[z,z']_{\lie{g}}\otimes ab,\ z,z'\in\fg,\ a,b\in A,$$
where $ [ , ]_{\lie{g}}$  denotes the super commutator bracket. $\lie{g}$ can be embedded in this Lie superalgebra as $\lie{g}\otimes 1$, see also \cite{BC,Sav}.

%$A$ will denote a fixed commutative associative algebra over $\C$. $\bb$ will denote a fixed %$\C$--basis of $A$. Write the $\C$ basis of the Lie algebra $\lie{sl}_2$ as $\{x^-,h,x^+\}$.  For %each $\alpha\in R$, let $\Omega_\alpha:\bu(h\otimes A)\to\bu(\fh\otimes A)$ be the algebra %homomorphism defined by
%$$h\otimes a\mapsto h_\alpha\otimes a$$

\section{An integral form and integral basis}

If $\mathcal{A}$ is an algebra over the field $\C$ of complex numbers, define an integral form $\mathcal{A}_\bbz$ of $\mathcal{A}$ to be a $\bbz$-algebra such that $\mathcal{A}_\bbz\otimes_\bbz\C=\mathcal{A}$. An integral basis for $\mathcal{A}$ is a $\bbz$-basis for $\mathcal{A}_\bbz$. In this section we define our integral form and state its integral basis as a theorem. %We also give an important corollary to this theorem.

\subsection{Multisets and $p(\chi)$}

Given any set $S$ define a \emph{multiset of elements of $S$} to be a multiplicity function $\chi:S\to\Z_{\geq0}$. Define $\f(S):=\{\chi:S\to\Z_{\geq0}:|\supp\chi|<\infty\}$. For $\chi\in\f(S)$ define $|\chi|:=\sum_{s\in S}\chi(s)$. Notice that $\f(S)$ is an abelian monoid under function addition. Define a partial order on $\f(S)$ so that for
$\psi,\chi\in\f(S)$, $\psi\leq\chi$ if $\psi(s)\leq\chi(s)$ for all $s\in S$. Define $\f_k(S):=\{\chi\in\f(S):|\chi|=k\}$ and given $\chi\in\f(S)$ define $\f(S)(\chi):=\{\psi\in\f(S):\psi\leq\chi\}$ and $\f_k(S)(\chi):=\{\psi\in\f(S)(\chi):|\psi|=k\}$. In the case $S=A$ the $S$ will be omitted from the notation. So that $\f:=\f(A)$, $\f_k:=\f_k(A)$, $\f(\chi):=\f(A)(\chi)$ and $\f_k(\chi):=\f_k(A)(\chi)$.

If $\psi\in\f(\chi)$ we define $\chi-\psi$ by standard function subtraction. Also define functions $\pi:\f-\{0\}\to A$ by
$$\pi(\psi):=\prod_{a\in A}a^{\psi(a)}$$
and $\pi(0)=1$, and $\m:\f\to\Z$ by
$$\m(\psi):=\frac{|\psi|!}{\prod_{a\in A}\psi(a)!}$$
For all $\psi\in\f$, $\m(\psi)\in\Z$ because if $\supp\psi=\{a_1,\ldots,a_k\}$ then $\m(\psi)$ is the multinomial coefficient
$$\binom{|\psi|}{\psi(a_1),\ldots,\psi(a_k)}$$

For $s\in S$ define $\chi_s$ to be the characteristic function of the set $\{s\}$. Then for all $\chi\in\f(S)$
$$\chi:=\sum_{s\in S}\chi(s)\chi_s$$

Given $\alpha\in R$ and $S\subset A$ define $X_\alpha:\f(S)\to\bu(\fg\otimes A)$ by
$$X_\alpha(\chi):=\prod_{a\in\supp\chi}\left(x_\alpha\otimes a\right)^{(\chi(a))}$$

%\subsection{Definition of $p(\chi)$}

Given $\chi\in\f$ and $\a\in R$, recursively define functions $p_{\a}:\f\to\bu(h_{\a}\otimes A)$ by $p_{\a}(0):=1$ and for $\chi\in\f-\{0\}$,
\begin{eqnarray*}
p_{\a}(\chi)&:=&-\frac{1}{|\chi|}\sum_{\psi\in\f(\chi)-\{0\}}\m(\psi)\left(h_{\a}\otimes\pi(\psi)\right)p_{\a}(\chi-\psi)
\end{eqnarray*}

For all $i\in I$, define $p_i(\chi):=p_{\alpha_i}(\chi)$.

\begin{rem}
\begin{enumerate}
\item[(1)] The $p_\alpha(\chi)$ are a generalization of Garland's $\Lambda_k\left(h_{\a}\otimes t^r\right)$ because
$p_\alpha\left(k\chi_{t^r}\right)=\Phi\left(\Lambda_{k-1}\left(h_{\a}\otimes t^r\right)\right)$, \cite[page 502]{Gar}.\\

\item[(2)] Given $\alpha\in R^+$ and $\psi\in\f$ \\ $X_\alpha\left(|\psi|\chi_1\right)X_{-\alpha}(\psi)\equiv(-1)^{|\psi|}p_\alpha(\psi)\mod(\bu(\fg\otimes A)(x_\alpha\otimes A))$, \cite[ Lemma 5.4]{C}.
\end{enumerate}
\end{rem}

We record some basic properties of $p_\alpha(\chi) $ for the rest of the paper.
\begin{prop}\label{degp}
Let $\alpha,\beta\in R$, $S\subset A$, $\chi,\varphi\in\f(S)$, $i \in I$,  and $b\in S$. Then
\begin{enumerate}
\item  $p_\alpha\left(\chi_b\right)=-\left(h_\alpha\otimes b\right)$\label{palpha1}\\
\item $p_\alpha(\chi)=(-1)^{|\chi|}\prod_{a\in A}(h_\alpha\otimes a)^{(\chi(a))}+ \textnormal{elements of }\bu(h_\alpha\otimes A)\textnormal{ of degree less than }|\chi|$\label{degpalpha}\\
\item $p_\alpha(\chi)p_\beta(\varphi)=p_\beta(\varphi)p_\alpha(\chi)$\label{palphapbeta}\\
\item $p_i(\chi)p_i(\varphi)=\prod_{a\in A}\binom{(\chi+\varphi)(a)}{\chi(a)}p_i(\chi+\varphi)+u$, where $u$ is in $\Z-\span\left\{p_i(\psi):\psi\in\f(S)\right\}$ with $\deg u<|\chi|+|\varphi|$
\end{enumerate}
\end{prop}
\begin{proof}
(1) is easily checked by direct calculation. (2) can be shown by induction on $|\chi|$ using (1) for the base case. (3) is true because $p_{\a}(\chi),p_{\beta}(\varphi)\in\bu(\fh\otimes A)$, which in turn is a subset of the center of $\bu(\fg\otimes A)$. (4) is proved as Lemma 5.1.2 in \cite{BC}.
\end{proof}

\begin{rem}\label{hbasis}
Note that, by induction on $|\chi|$, Proposition \ref{degp}\eqref{degpalpha} tells us that if $\bb$ is a $\C$-basis for $A$ then $\left\{p_\alpha(\chi):\chi\in\f(\bb)\right\}$ is a $\C$-basis for $\bu\left(\left\{h_\alpha\right\}\otimes A\right)$ and furthermore that the set of all products of elements of the set $\{p_i(\chi)\ |\ i\in I,\ \chi\in\f(\bb)\}$ is a $\C$-basis for $\bu\left(\fh\otimes A\right)$.
\end{rem}

\subsection{}
Assume that $A$ has a basis $\bb$, which is closed under multiplication.
\begin{defn}
 Define the integral form  $\bu_\Z(\fg\otimes A)$  to be the $\Z$-subalgebra of $\bu(\fg\otimes A)$ generated by the following
$$(x_{\a,k}\otimes b)^{(s)}; \  \  \a\in R_{\0},  \ k\in[\mu(\a)], \  b \in \bb, \  s\in\Z_{\geq0}$$
$$x_{\gamma,n}\otimes c;  \  \  \gamma\in R_{\1}, \  n\in[\mu(\gamma)], \ c \in \bb$$
$$p_i(\chi); \ \  i\in I, \ \chi\in\f(\bb)$$
\end{defn}

Define a \emph{monomial} in $\bu_\Z(\fg\otimes A)$ to be any finite product of elements of the generating set of $\bu_\Z(\fg\otimes A)$. Given a monomial $m$ its \emph{factors} are elements of the generating set of $\bu_\Z(\fg\otimes A)$ appearing in $m$.

The goal of this paper is to prove the following theorem, which is the analog of Theorem 3.1.3 in \cite{BC}.

\begin{thm}\label{thm}
The $\Z$-superalgebra $\bu_\Z(\fg\otimes A)$ is a free $\Z$-module. Let $\left(\preccurlyeq,R\cup I\right)$ be a total order. Then a $\Z$ basis of $\bu_\Z(\fg\otimes A)$ is given by the set $\mathcal{B}$ of all products (without repetitions) of elements of the set
$$\left\{\left(x_{\a,k}\otimes a_{\a,k}\right)^{(r_{\a,k})},p_i(\psi_i),\left(x_{\gamma,n}\otimes c_{\gamma,n}\right)\right\}$$
where $\a\in R_{\0}$, $k\in[\mu(\a)]$, $i\in I$, $\gamma\in R_{\1}$, $n\in[\mu(\gamma)]$, $\psi_i\in\f(\bb)$, $a_{\alpha,k},c_{\gamma,n}\in\bb$, and $r_{\alpha,k}\in\Z_{\geq0}$
taken in the order given by $\left(\preccurlyeq,R\cup I\right)$.
\end{thm}

\begin{rems}\label{bborder}
\begin{enumerate}
\item In the case where $\gamma\in R$ and $2\gamma\in R$ (which is described in Proposition \ref{rootprop}(2)) we will also need a total order $\precsim$ on the basis $\bb$ for $A$ in order to state our integral basis for $\bu_\bbz(\fg\otimes A)$ because in this case $(x_\gamma\otimes a)$ and $(x_\gamma\otimes b)$ do not commute. The products in $\B$ will be taken first in the order $\left(\preccurlyeq,R\cup I\right)$ and then in the order given by $(\precsim,\bb)$.

\item If $\fg$ is of type $S(n)$ or $H(n)$, so that $\fg\subsetneq\overline{\fg}$, we add the element $h_\delta$ to our Chevalley basis of $\fg$ to obtain a Chevalley basis for $\overline{\fg}$. All of the results in this work can then be applied to $\overline{\fg}$ by including $h_{\delta}$ as necessary.
\end{enumerate}
\end{rems}

\begin{cor} Let $\fg_+$, $\fg_{\0}^+$, $\fg_{\1}^+$ denote positive parts of $\fg$ , $\fg_{\0} $ and $\fg_{\1}$, respectively, in the $\Z$-grading. We have the following isomorphisms of $\Z$-modules.
\begin{itemize}
 \item[(1)]$\bu_\Z(\fg\otimes A)  \cong  \bu_\Z(\fg_{\0}\otimes A) \otimes_{\Z} \Lambda_{\Z}(\fg_{\1} \otimes A).$
 \item[(2)]$\bu_\Z((\fg_0\oplus \fg_+)\otimes A)\cong \bu_\Z(\fg_{\0}\otimes A) \otimes_{\Z} \Lambda_{\Z}(\fg_{\1}^+ \otimes A).$
\item[(3)]$\bu_\Z(\fg_{\0}\otimes A)\cong \bu_\Z(\fg_{0}\otimes A) \otimes_{\Z} \Lambda_{\Z}(\fg_{\0}^+ \otimes A).$
\end{itemize}
\end{cor}
\begin{proof}
(1) can be proved similarly to Corollary 3.1.5 in \cite{BC}. (2) is proved similarly but we apply Theorem \ref{thm} to choose a basis $\mathcal{B}$ of $\bu_\Z((\fg_0\oplus \fg_+)\otimes A)$ where $\preccurlyeq$ is such that $(R_{\0}\cup I)\preccurlyeq\bigcup_{z\in\Z_{\geq0}}R_{2z+1} $. (3) is again similar; this time we apply the theorem to $\bu_\Z(\fg_{\0}\otimes A)$ and the order should be such that $(R_0\cup I)\preccurlyeq\bigcup_{z\in\N}R_{2z}$.
\end{proof}

\section{Straightening Identities}

In this section we state all commutation relations and provide a proofs for some of these identities. We have commutation relations involving even generators and those involving even and odd generators. The proofs will follow the lists of the straightening identities. Note that in all of these identities all of the coefficients are in $\Z$. Also, note that in the cases which are not listed the factors  either commute or anticommute.
%In this section we state and prove all necessary commutation relations. The others either %commute or anticommute.

\subsection{Even Generator Straightening Identities}

Given $\chi\in\f(S)$ define $\supp\chi:=\{s\in S:\chi(s)>0\}$.

Given $\chi\in\f$ define
$$\cs(\chi):=\left\{\psi\in\f(\f):\sum_{\phi\in\f}\psi(\phi)\phi\leq\chi\right\}\textnormal{ and }\cs_k(\chi):=\cs(\chi)\cap\f_k(\f)$$
Given $j\geq0$ define
$$\cp(j):=\left\{\lambda\in\f\left(\Z_{\geq0}\right):\sum_{m\in\supp\lambda}\lambda(m)m=j\right\}
\textnormal{ and }\cp_k(j):=\cp(j)\cap\f_k\left(\Z_{\geq0}\right)$$
For all $j,k\in\Z_{\geq0}$, $\a\in R$ and $c,d\in A$, define $D^{\a,1}_{j,k}(0,d):=0=:D^{\a,1}_{j,k}(c,0)$ and $D^{\a,1}_{j,k}:(A-\{0\})^2\to\bu(\fg\otimes A)$ by
$$D^{\a,1}_{j,k}(d,c):=\sum_{\lambda\in\cp_k(j)}\prod_{m\in\supp\lambda}\left(x_{\a,1}\otimes d^mc\right)^{(\lambda(m))}$$

\begin{rem}\label{D}
%\begin{enumerate}
%\item[(a)]$ D_{0,k}^{\a,1}(d,c)=\left(x_{\a,1}\otimes c\right)^{(k)}$.
%\item $ D_{j,0}^{\a,1}(d,c)=\delta_{j,0}$.
%\item
The definition of $D_{j,k}^{\a,1}(d,c)$ above is the analog of the recursive definition of $D_{\a}^\pm\left(j\chi_1,j\chi_d,k\chi_c\right)$ from \cite{C} by Proposition 5.2 in \cite{C}.
%\end{enumerate}
\end{rem}

Define $x_{\a,k}=0$ if $\a\notin R$. The notation below is from the definition of the Chevalley basis.

\begin{prop}\label{strteven}
For all $\eta\in R_{\0}$, $u\in[\mu(\eta)]$, $i, j  \in I$, $\alpha,\zeta\in R_0$, $\beta,\gamma\in R_{\0}$, with $\beta+\gamma\in R$ and $2\beta+\gamma,\beta+2\gamma\notin R$, $m\in[\mu(\beta)]$, $n\in[\mu(\gamma)]$, $\vartheta\in R_{\0}-R_0$, $q\in[\mu(\vartheta)]$, $\chi,\varphi\in\f$, $a,b\in A$, and $r,s\in\Z_{\geq0}$. Where $v^\prime\in[\mu(2\a+\vartheta)]$ and $\varepsilon\in\{\pm1\}$ are such that $[x_{\a,1}[x_{\a,1},x_{\vartheta,q}]]=2\varepsilon x_{2\a+\vartheta,v^\prime}$.
\begin{eqnarray}
p_i(\chi)p_j(\varphi)&=&p_j(\varphi)p_i(\chi)\label{pipj}\\
\left(x_{\eta,u}\otimes a\right)^{(r)}\left(x_{\eta,u}\otimes a\right)^{(s)}
&=&\binom{r+s}{s}\left(x_{\eta,u}\otimes a\right)^{(r+s)}\label{xaxa}\\
\left(x_{\alpha,1}\otimes a\right)^{(r)}\left(x_{-\alpha,1}\otimes b\right)^{(s)}
&=&\sum_{\substack{j,k,v\in\Z_{\geq0}\\j+k+v\leq\min(r,s)}}(-1)^{j+k+v}D^{-\alpha,1}_{j,s-j-k-v}(ab,b)p_{\alpha}\left(k\chi_{ab}\right)\nonumber\\
&\times&D^{\alpha,1}_{v,r-j-k-v}(ab,a)\label{x+x-}\\
\left(x_{\eta,u}\otimes a\right)^{(r)}p_i(\chi)
&=&\sum_{\psi\in\cs_r(\chi)}p_i\left(\chi-\sum_{\phi\in\f}\psi(\phi)\phi\right)\nonumber\\
&\times&\prod_{\varphi\in\f}\left(\binom{\eta(h_i)+|\varphi|-1}{|\varphi|}\m(\varphi)\left(x_{\eta,u}\otimes a\pi(\varphi)\right)\right)^{(\psi(\varphi))}\label{xrpi}
\end{eqnarray}
\begin{eqnarray}
p_i(\chi)\left(x_{\eta,u}\otimes a\right)^{(r)}
&=&\sum_{\psi\in\cs_r(\chi)}\prod_{\varphi\in\f}\left(\binom{-\eta(h_i)+|\varphi|-1}{|\varphi|}
\m(\varphi)\left(x_{\eta,u}\otimes a\pi(\varphi)\right)\right)^{(\psi(\varphi))}\nonumber\\
&\times&p_i\left(\chi-\sum_{\phi\in\f}\psi(\phi)\phi\right)\label{pixr}\\
\left(x_{\beta,m}\otimes a\right)^{(r)}\left(x_{\gamma,n}\otimes b\right)^{(s)}
&=&\sum_{j=0}^{\min(r,s)}\left(x_{\gamma,n}\otimes b\right)^{(s-j)}\left(\sum_{\psi\in\f_j([\mu(\beta+\gamma)])}\prod_{v=1}^{\mu(\beta+\gamma)}\left(c_{\beta,m}^{\gamma,n}(v)\right)^{\psi(v)}\left(x_{\beta+\gamma,v}\otimes ab\right)^{(\psi(v))}\right)\nonumber\\
&\times&\left(x_{\beta,m}\otimes a\right)^{(r-j)}\label{xbmarxgnbs}\\
\left(x_{\alpha,1}\otimes a\right)^{(r)}\left(x_{\vartheta,q}\otimes b\right)^{(s)}
&=&\sum_{\substack{j_1,j_2\in\Z_{\geq0}\\j_1+j_2\leq s\\j_1+2j_2\leq r}}\left(x_{\vartheta,q}\otimes b\right)^{(s-j_1-j_2)}\varepsilon^{j_2}\left(x_{2\alpha+\vartheta,v^\prime}\otimes a^2b\right)^{(j_2)}\nonumber\\
&\times&\left(\sum_{\psi\in\f_{j_1}([\mu(\a+\vartheta)])}\prod_{v=1}^{\mu(\alpha+\vartheta)}\left(c_{\alpha,1}^{\vartheta,q}(v)\right)^{\psi(v)}\left(x_{\alpha+\vartheta,v}\otimes ab\right)^{(\psi(v))}\right)\nonumber\\
&\times&\left(x_{\alpha,1}\otimes a\right)^{(r-j_1-2j_2)}\label{xa1rxtqs}\\
\left(x_{\alpha,1}\otimes a\right)^{(r)}\left(x_{\zeta,1}\otimes b\right)^{(s)}
&=&\sum_{\substack{\psi\in\f\left(\N^2\right)\\\sum j\psi(j,k)\leq r\\\sum k\psi(j,k)\leq s}}\varepsilon_\psi\left(x_{\zeta,1}\otimes b\right)^{\left(s-\sum k\psi(j,k)\right)}\prod_{(j,k)\in\supp\psi}\left(x_{j\alpha+k\zeta,1}\otimes a^jb^k\right)^{(\psi(j,k))}
\nonumber\\
&\times&\left(x_{\alpha,1}\otimes a\right)^{\left(r-\sum j\psi(j,k)\right)}\label{xaxb}
\end{eqnarray}
\end{prop}
Where $\varepsilon_\psi\in\{\pm1\}$ for all $\psi\in\f\left(\N^2\right)$, $x_{j\gamma+k\beta}=0$ if $j\gamma+k\beta\notin R$ and all of the unlabeled sums are over all $(j,k)\in\supp\psi$.
%Note: All others commute
\begin{proof}
\eqref{pipj} is clear because for all $\alpha\in R$ and all $\chi\in\f$ $p_{\a}(\chi)$ is in the center of $\bu(\fg\otimes A)$. \eqref{xaxa} is easily shown by direct calculation. \eqref{x+x-} is a slight variation of a special case of Lemma 5.4 in \cite{C}. Lemma 5.4 in \cite{C} applies because for all $\alpha\in R_0$ $\left\{x_{-\alpha,1},h_\alpha,x_{\alpha,1}\right\}$ is an $\lie{sl}_2$-triple. Proposition 5.2 in \cite{C} gives the necessary equivalence of our nonrecursive definition of $D_{j,k}^{\pm\alpha}(d,c)$ and the more general recursive definition in \cite{C}, (see Remark \ref{D}). \eqref{xrpi} is proved by induction on $r$. The base case $(r=1)$ is proved in Lemma 4.1.4 in \cite{BC}. Assuming the base case we have
\begin{eqnarray*}
(r+1)\left(x_{\eta,u}\otimes a\right)^{(r+1)}p_i(\chi)
&=&\left(x_{\eta,u}\otimes a\right)^{(r)}\left(x_{\eta,u}\otimes a\right)p_i(\chi)\\
&=&\left(x_{\eta,u}\otimes a\right)^{(r)}\sum_{\psi^\prime\in\f(\chi)}p_i(\chi-\psi^\prime)\binom{\eta(h_i)+|\psi^\prime|-1}{|\psi^\prime|}\m(\psi^\prime)\left(x_{\eta,u}\otimes a\pi(\psi^\prime)\right)\\
&=&\sum_{\psi^\prime\in\f(\chi)}\binom{\eta(h_i)+|\psi^\prime|-1}{|\psi^\prime|}\m(\psi^\prime)\sum_{\tau\in\cs_r(\chi-\psi^\prime)}p_i\left(\chi-\psi^\prime-\sum_{\phi\in\f}\tau(\phi)\phi\right)\nonumber\\
&\times&\prod_{\varphi\in\f}\left(\binom{\eta(h_i)+|\varphi|-1}{|\varphi|}\m(\varphi)\left(x_{\eta,u}\otimes a\pi(\varphi)\right)\right)^{(\tau(\varphi))}\left(x_{\eta,u}\otimes a\pi(\psi^\prime)\right)\\
&&(\textnormal{by the induction hypothesis})\\
&=&\sum_{\psi^\prime\in\f(\chi)}\sum_{\tau\in\cs_r(\chi-\psi^\prime)}p_i\left(\chi-\sum_{\phi\in\f}\left(\tau+\chi_{\psi^\prime}\right)(\phi)\phi\right)(\tau(\psi^\prime)+1)\nonumber\\
&\times&\prod_{\varphi\in\f}\left(\binom{\eta(h_i)+|\varphi|-1}{|\varphi|}\m(\varphi)\left(x_{\eta,u}\otimes a\pi(\varphi)\right)\right)^{\left(\left(\tau+\chi_{\psi^\prime}\right)(\varphi)\right)}\\
&=&\sum_{\psi\in\cs_{r+1}(\chi)}\sum_{\psi^\prime\in\f}\psi(\psi^\prime)p_i\left(\chi-\sum_{\phi\in\f}\psi(\phi)\phi\right)\nonumber\\
&\times&\prod_{\varphi\in\f}\left(\binom{\eta(h_i)+|\varphi|-1}{|\varphi|}\m(\varphi)\left(x_{\eta,u}\otimes a\pi(\varphi)\right)\right)^{\left(\psi(\varphi)\right)}\\
&=&(r+1)\sum_{\psi\in\cs_{r+1}(\chi)}p_i\left(\chi-\sum_{\phi\in\f}\psi(\phi)\phi\right)\nonumber\\
&\times&\prod_{\varphi\in\f}\left(\binom{\eta(h_i)+|\varphi|-1}{|\varphi|}\m(\varphi)\left(x_{\eta,u}\otimes a\pi(\varphi)\right)\right)^{\left(\psi(\varphi)\right)}
\end{eqnarray*}
So we have proved \eqref{xrpi}. \eqref{pixr} is proved similarly and hence the proof is omitted. \eqref{xbmarxgnbs} and \eqref{xa1rxtqs} are proved using Lemma 4.1.5(1) and (2) in \cite{BC}, respectively, and the Multinomial Theorem. \eqref{xaxb} is Proposition 4.1.2  equation (4.1.6) in \cite{BC}.
\end{proof}

\subsection{Odd and even generator straightening identities}

Recall that $R_{-1}$ is the set  of roots of height $-1$. %for $\fg_{-1}$ as a $\fg_0$-module.

\begin{prop}\label{strtodd}
For all $\alpha\in R_{\0}$, $\beta\in R_0$, $\gamma,\zeta\in R_{\1}$, $\vartheta\in R_{-1}$, $k\in[\mu(-\vartheta)]$, $m\in[\mu(\gamma)]$, $n\in[\mu(\zeta)]$, $q\in[\mu(\alpha)]$, $r\in\N$, $a,b\in A$, $i\in I$, $\chi\in\f$ Where $s\in[\mu(2\beta+\gamma)]$ and $\varepsilon\in\{\pm1\}$ are such that $\left[x_{\beta,1},\left[x_{\beta,1},x_{\gamma,1}\right]\right]=2\varepsilon x_{2\beta+\gamma,s}$.
\begin{eqnarray}
\left(x_{\gamma,m}\otimes a\right)p_i(\chi)
&=&\sum_{\psi\in\f(\chi)}\binom{|\psi|-1+\gamma(h_i)}{|\psi|}\m(\psi)p_i(\chi-\psi)\left(x_{\gamma,m}\otimes a\pi(\psi)\right)\label{xgammapi}\\
p_i(\chi)\left(x_{\gamma,m}\otimes a\right)
&=&\sum_{\psi\in\f(\chi)}\binom{|\psi|-1-\gamma(h_i)}{|\psi|}\m(\psi)\left(x_{\gamma,m}\otimes a\pi(\psi)\right)p_i(\chi-\psi)\label{pixgamma}\\
\left(x_{\gamma,m}\otimes a\right)\left(x_{\gamma,m}\otimes b\right)
&=&\left\{\begin{array}{cc}\pm\left(x_{2\gamma,j^\prime}\otimes ab\right)&\textnormal{if }c_{\gamma,m}^{\gamma,m}(j^\prime)=\pm2\textnormal{ for some }j^\prime\in[\mu(2\gamma)]\\0&\textnormal{otherwise}\end{array}\right.\label{2gam}\\
\left(x_{\vartheta,1}\otimes a\right)\left(x_{-\vartheta,k}\otimes b\right)
&=&-\left(x_{-\vartheta,k}\otimes b\right)\left(x_{\vartheta,1}\otimes a\right)+\left(h_{\sigma_\vartheta(k)}\otimes ab\right)\label{xdel1x-delk}\\
\left(x_{\gamma,m}\otimes a\right)\left(x_{\zeta,n}\otimes b\right)
&=&-\left(x_{\zeta,n}\otimes b\right)\left(x_{\gamma,m}\otimes a\right)
+\sum_{j=1}^{\mu(\gamma+\zeta)}c_{\gamma,m}^{\zeta,n}(j)\left(x_{\gamma+\zeta,j}\otimes ab\right)\label{xgammxzetn}\\
\left(x_{\alpha,q}\otimes a\right)^{(r)}\left(x_{\gamma,m}\otimes b\right)
&=&\left(x_{\gamma,m}\otimes b\right)\left(x_{\alpha,q}\otimes a\right)^{(r)}\nonumber\\ &+&\sum_{j=1}^{\mu(\alpha+\gamma)}c_{\alpha,q}^{\gamma,m}(j)\left(x_{\alpha+\gamma,j}\otimes ab\right)\left(x_{\alpha,q}\otimes a\right)^{(r-1)}\label{xalqrxgamm}\\
\left(x_{\beta,1}\otimes a\right)^{(r)}\left(x_{\gamma,m}\otimes b\right)
&=&\left(x_{\gamma,m}\otimes b\right)\left(x_{\beta,1}\otimes a\right)^{(r)}
+\sum_{j=1}^{\mu(\beta+\gamma)}c_{\beta,1}^{\gamma,m}(j)\left(x_{\beta+\gamma,j}\otimes ab\right)\left(x_{\beta,1}\otimes a\right)^{(r-1)}\nonumber\\
&+&\varepsilon\left(x_{2\beta+\gamma,s}\otimes a^2b\right)\left(x_{\beta,1}\otimes a\right)^{(r-2)}\label{xbeta1rxgamm}
\end{eqnarray}
\end{prop}
\begin{proof}
\eqref{xgammapi} and \eqref{pixgamma} are proved in Lemma 4.1.4 in \cite{BC}. \eqref{2gam}, \eqref{xdel1x-delk}, and \eqref{xgammxzetn} are all shown by direct calculation. \eqref{xalqrxgamm} and \eqref{xbeta1rxgamm} can be shown by induction on $r$.
\end{proof}

%Note: All others commute or anticommute.

\section{Proof of Theorem \ref{thm}}

In this section we will prove Theorem \ref{thm} and give a triangular decomposition of $\bu_\Z(\fg\otimes A)$. The proof will proceed by induction on the degree of monomials in $\bu_\Z(\fg\otimes A)$ (using the definition of degree in Section \ref{defdeg}) and the following lemmas.

\subsection{}

\begin{lem}\label{brack}
For all $\a\in R_0$, $\beta,\gamma\in R_{\0}$, with $(\beta\neq-\gamma)$, $k\in[\mu(\beta)]$, $m\in[\mu(\gamma)]$, $\zeta,\eta\in R_{\1}$, $n\in[\mu(\zeta)]$, $q\in[\mu(\eta)]$, $\chi\in\f(\bb)$, $r,s\in\Z_{\geq0}$, $a,b\in\bb$, and $i\in I$.
\begin{itemize}
\item [(1)] $\left[\left(x_{\a,1}\otimes a\right)^{(r)},\left(x_{-\a,1}\otimes b\right)^{(s)}\right]$ has degree less than $r+s$.\\

\item [(2)] $\left[\left(x_{\beta,k}\otimes a\right)^{(r)},p_i(\chi)\right]$ has degree less than $r+|\chi|$.\\

\item [(3)] $\left[\left(x_{\beta,k}\otimes a\right)^{(r)},\left(x_{\gamma,m}\otimes b\right)^{(s)}\right]$ has degree less than $r+s$.\\

\item [(4)] $\left[\left(x_{\zeta,n}\otimes a\right),p_i(\chi)\right]$ has degree less than $|\chi|+1$.\\

\item [(5)] $\left[\left(x_{\beta,k}\otimes a\right)^{(r)},\left(x_{\zeta,n}\otimes b\right)\right]$ has degree less than $r+1$.\\

\item [(6)] $\left[\left(x_{\zeta,n}\otimes a\right),\left(x_{\eta,q}\otimes b\right)\right]\in\Z$--span $\B$ and has degree less than $2$.
\end{itemize}
\begin{proof}
For $(1)$ by \eqref{x+x-} we have
\begin{eqnarray*}
\left(x_{\a,1}\otimes a\right)^{(r)}\left(x_{-\a,1}\otimes b\right)^{(s)}
&=&\sum_{\substack{j,k,v\in\Z_{\geq0}\\0<j+k+v\leq\min(r,s)}}(-1)^{j+k+v}D^{-\a,1}_{j,s-j-k-v}(ab,b)p_{\a}\left(k\chi_{ab}\right)D^{\a,1}_{v,r-j-k-v}(ab,a)\\
&+&\left(x_{-\a,1}\otimes b\right)^{(s)}\left(x_{\a,1}\otimes a\right)^{(r)}\hskip.2in\left(\textnormal{by definition } D^{\vartheta,1}_{0,t}(d,c)=\left(x_{\vartheta,1}\otimes c\right)^{(t)}\right)
\end{eqnarray*}
By the definition of $D^{\vartheta,1}_{j,t}(d,c)$, $D^{-\a,1}_{j,s-j-k-v}(ab,b)p_{\a}\left(k\chi_{ab}\right)D^{\a,1}_{v,r-j-k-v}(ab,a)\in\Z$--span $\B$, and either the sum on the right-hand side is zero or its terms have degree $r+s-2j-k-2v<r+s$.

For $(2)$ by \eqref{xrpi} we have
\begin{eqnarray*}
&&\hskip-.4in\left(x_{\beta,k}\otimes a\right)^{(r)}p_i(\chi)\\
&=&\sum_{\psi\in\cs_r(\chi)-\{r\chi_0\}}p_i\left(\chi-\sum_{\phi\in\f}\psi(\phi)\phi\right)\prod_{\varphi\in\f}\left(\binom{|\varphi|-1+\beta(h_i)}{|\varphi|}\m(\varphi)\left(x_{\beta,k}\otimes a\pi(\varphi)\right)\right)^{(\psi(\varphi))}\\
&+&p_i(\chi)\left(x_{\beta,k}\otimes a\right)^{(r)}
\end{eqnarray*}
and by \eqref{pixr} we have
\begin{eqnarray*}
&&\hskip-.4in p_i(\chi)\left(x_{\beta,k}\otimes a\right)^{(r)}\\
&=&\sum_{\psi\in\cs_r(\chi)-\{r\chi_0\}}\prod_{\varphi\in\f}\left(\binom{|\varphi|-1-\beta(h_i)}{|\varphi|}\m(\varphi)\left(x_{\beta,k}\otimes a\pi(\varphi)\right)\right)^{(\psi(\varphi))}p_i\left(\chi-\sum_{\phi\in\f}\psi(\phi)\phi\right)\\
&+&\left(x_{\beta,k}\otimes a\right)^{(r)}p_i(\chi)
\end{eqnarray*}
In one of the cases the factors are in the order specified  by $(\preccurlyeq,R\cup I)$ and either the sums on the right-hand sides are zero or by Proposition \ref{degp} the sums on the right-hand sides have degree
$$\left|\chi-\sum_{\phi\in\f}\psi(\phi)\phi\right|+|\psi|=|\chi|+r-\sum_{\phi\in\f}\psi(\phi)|\phi|<|\chi|+r$$
since $\psi\neq r\chi_0$.

For $(3)$ we use \eqref{xbmarxgnbs}, \eqref{xa1rxtqs}, and \eqref{xaxb}. Using \eqref{xaxb} we have
\begin{eqnarray*}
\left(x_{\alpha,1}\otimes a\right)^{(r)}\left(x_{\zeta,1}\otimes b\right)^{(s)}
&=&\sum_{\substack{\psi\in\f\left(\N^2\right)-\{0\}\\\sum j\psi(j,k)\leq r\\\sum k\psi(j,k)\leq s}}\varepsilon_\psi\left(x_{\zeta,1}\otimes b\right)^{\left(s-\sum k\psi(j,k)\right)}\prod_{(j,k)\in\supp\psi}\left(x_{j\alpha+k\zeta,1}\otimes a^jb^k\right)^{(\psi(j,k))}
\nonumber\\
&\times&\left(x_{\alpha,1}\otimes a\right)^{\left(r-\sum j\psi(j,k)\right)}+\left(x_{\zeta,1}\otimes b\right)^{(s)}\left(x_{\alpha,1}\otimes a\right)^{(r)}
\end{eqnarray*}
Either the sum on the right-hand side is zero or its terms have degree
$$r+s-|\psi|-\sum_{(j,k)\in\supp\psi}\left(j\psi(j,k)+k\psi(j,k)\right)<r+s$$

For (4) by \eqref{xgammapi} we have
\begin{eqnarray*}
\left(x_{\zeta,n}\otimes a\right)p_i(\chi)
&=&\sum_{\psi\in\f(\chi)-\{0\}}\binom{|\psi|-1+\zeta(h_i)}{|\psi|}\m(\psi)p_i(\chi-\psi)\left(x_{\zeta,n}\otimes a\pi(\psi)\right)+p_i(\chi)\left(x_{\zeta,n}\otimes a\right)
\end{eqnarray*}
and by \eqref{pixgamma} we have
\begin{eqnarray*}
p_i(\chi)\left(x_{\zeta,n}\otimes a\right)
&=&\sum_{\psi\in\f(\chi)-\{0\}}\binom{|\psi|-1-\zeta(h_i)}{|\psi|}\m(\psi)\left(x_{\zeta,n}\otimes a\pi(\psi)\right)p_i(\chi-\psi)+\left(x_{\zeta,n}\otimes a\right)p_i(\chi)
\end{eqnarray*}
In one of the cases the factors are in the order specified  by $(\preccurlyeq,R\cup I)$ and either the sums on the right-hand sides are zero or by Proposition \ref{degp} the sums on the right-hand sides have degree $|\chi|-|\psi|+1<|\chi|+1$.
For (5) and (6)  just apply \eqref{xalqrxgamm}, \eqref{xbeta1rxgamm}, \eqref{2gam}, \eqref{xdel1x-delk} and \eqref{xgammxzetn} as needed.
\end{proof}
\end{lem}

%The following Lemma is proved as Lemma 5.1.2 in \cite{BC}.
%\begin{lem}\label{degpi}
%For all $i\in I$ and $\chi,\varphi\in\f(\bb)$
%$$p_i(\chi)p_i(\varphi)=\prod_{a\in A}\binom{(\chi+\varphi)(a)}{\chi(a)}p_i(\chi+\varphi)+u$$
%where $u$ is in $\Z-\span\left\{p_i(\psi):\psi\in\f(\bb)\right\}$ with $\deg u<|\chi|+|\varphi|$.
%\end{lem}

\noindent\emph{Proof of Theorem \ref{thm}.} Fix an order $(\preccurlyeq,R\cup I)$. The Poincar\'{e}-Birkhoff-Witt Theorem for Lie superalgebras and Remark \ref{hbasis} imply that $\B$ is a $\C$--linearly independent set. Hence $\B$ is a $\bbz$--linearly independent set.

It suffices to show that the integer span of $\B$ is $\bu_\Z(\fg\otimes A)$. We will prove this by induction on the degree of monomials in $\bu_\Z(\fg\otimes A)$  (using the definition of degree in Section \ref{defdeg}). Since $p_i(\chi_a)=-(h_i\otimes a)$ any degree one monomial is in the integer span of $\B$. Let $m$ be any monomial in $\bu_\Z(\fg\otimes A)$. Recall that a monomial is any finite product of elements of the generating set of $\bu_\Z(\fg\otimes A)$.
If $m\in\B$ then we are done. If not then either the factors of $m$ are not in the order specified by $\preccurlyeq$, or $m$ has products of factors with the following forms
\begin{eqnarray}
\left(x_{\a,k}\otimes b\right)^{(r)}\left(x_{\a,k}\otimes b\right)^{(s)}&,&\a\in R_{\0},\ k\in[\mu(\a)],\ b\in\bb,\ r,s\in\Z_{\geq0}\label{facxbxb}\\
p_i(\chi)p_i(\psi)&,&i\in I,\ \chi,\psi\in\f(\bb)\label{facpipi}\\
\left(x_{\gamma,m}\otimes c\right)^{j}&,&\gamma\in R_1,\ m\in[\mu(\gamma)],\ c\in\bb,\ j\in\Z_{\geq0}\label{facxgam}.
\end{eqnarray}
If the factors of $m$ are not in the order given by $\preccurlyeq$ then we can order the factors of $m$ according to $\preccurlyeq$ using the straightening identities in Propositions \ref{strteven} and \ref{strtodd}. Lemma \ref{brack} implies that these rearrangements will only produce integer linear combinations of monomials with lower degree. These lower degree monomials are in the integer span of $\B$ by the induction hypothesis.

If (possibly after rearranging factors as above) $m$ contains products of factors as in \eqref{facxbxb} we apply \eqref{xaxa} to combine these pairs of factors into single factors with integral coefficients. If $m$ contains products of factors as in \eqref{facpipi} we apply Proposition \ref{degp}(4) to obtain an integer linear combination of single factors. Crucially all of these factors other than the leading term have lower degree and hence any monomial containing them is in the integer span of $\B$ by the induction hypothesis.  If $m$ contains powers of factors as in \eqref{facxgam} we apply \eqref{2gam} to get
$$\left(x_{\gamma,m}\otimes c\right)^{j}=\left(x_{\gamma,m}\otimes c\right)^{2k+\varepsilon}=z\left(x_{2\gamma,j^\prime}\otimes c^2\right)^{k}\left(x_\gamma\otimes c\right)^{\varepsilon}=zk!\left(x_{2\gamma,j^\prime}\otimes c^2\right)^{(k)}\left(x_\gamma\otimes c\right)^{\varepsilon}$$
where $z\in\{\pm1\}$, $k\in\Z_{\geq0}$, $\varepsilon\in\{0,1\}$. This monomial has lower degree and hence is in the integer span of $\B$ by the induction hypothesis.

We may need to repeat one or more of these steps but in the end we see that $m\in\Z$-$\span\ \B$. Thus the integer span of $\B$ is $\bu_\Z(\fg\otimes A)$ and so $\B$ is an integral basis for $\bu_\Z(\fg\otimes A)$.

If for some $\gamma\in R$, $2\gamma\in R$ we need a total order $(\precsim,\bb)$ (see Remark \ref{bborder}(1)). Using the identity \eqref{2gam} with we can reorder the products of factors of the form $\left(x_{\gamma,m}\otimes a\right)\left(x_{\gamma,m}\otimes b\right)$ as required by $\precsim$. With each reordering new monomials will be created, however, by Lemma \ref{brack}(6), they will have lower degree and hence will be in the integer span of $\B$ by the induction hypothesis. \hskip5in $\square$

\subsection{}

Let  $\bu_{\Z}^\pm(\fg \otimes A)$ denote the $\Z$-subalgebras of $\bu_{\Z}(\fg \otimes A)$ generated, respectively, by
$$\{(x_{\alpha}\otimes a)^{(r)},x_\gamma\otimes c:\alpha\in R_{\0}^\pm,\ \gamma\in R_{\1}^\pm,\ a,c\in\bb,\ r\in\Z_{\geq 0}\}.$$
Let $\bu_{\Z}^{0}(\fg \otimes A)$  denote the $\Z$-subalgebra of $\bu_{\Z}(\fg\otimes A)$ generated by
$$\{p_i(\chi):\chi\in\mathcal{F}(\bb),i\in I\}.$$

Then as a corollary to Theorem \ref{thm} we obtain the following triangular decomposition of $\bu_\Z(\fg\otimes A)$.

\begin{cor}
$$\bu_{\Z}(\fg \otimes A)=\bu_{\Z}^{-}(\fg \otimes A)\bu_{\Z}^{0}(\fg \otimes A)\bu_{\Z}^{+}(\fg \otimes A)$$
\end{cor}
\begin{proof}
In Theorem \ref{thm} choose $(\preccurlyeq,R)$ so that $R^-\preccurlyeq I\preccurlyeq R^+$. Then by Theorem \ref{thm} we can write every element of $\bu_{\Z}(\fg \otimes A)$ as a $\Z$-linear combination of elements of $\bu_{\Z}^{-}(\fg\otimes A)\bu_{\Z}^{0}(\fg\otimes A)\bu_\Z^{+}(\fg\otimes A)$.
\end{proof}

Let $\left(\preccurlyeq_\pm,R^\pm\right)$ be any total orders. Define $\B^\pm$ to be the sets of all products (without repetitions) of elements of the set
$$\left\{\left(x_\alpha\otimes b\right)^{(r_\alpha)},x_\gamma\otimes c\ |\ \alpha\in R_{\0}^\pm,\ \gamma\in R_{\1}^\pm,\ b,c\in\bb,\ r_\alpha\in\Z_{\geq0}\right\}$$
taken in the orders given by $\preccurlyeq_\pm$. Given a total order $(\preccurlyeq,I)$ define
$$\B^0:=\left\{\prod_{i\in I}p_i(\varphi_i):\varphi_1,\ldots,\varphi_l\in\f(\bb)\right\}$$
with the products taken in the order given by $\preccurlyeq$.

\begin{prop}
Let  $\mathcal{B}^\pm$ and $\mathcal{B}^0$ be as above. Then
\begin{itemize}
\item [(1)] $\mathcal{B}^\pm$ are $\Z$ bases for $\bu_{\Z}^{\pm}(\fg \otimes A)$ respectively.
\item [(2)] $\mathcal{B}^{0}$ is a $\Z$ basis for $\bu_{\Z}^{0}(\fg \otimes A)$.
\end{itemize}
\end{prop}
\begin{proof}
The proof of (1) is similar to that of Theorem \ref{thm}. In this cas we proceed by induction on the degree of a monomial in $\bu_{\Z}^{\pm}(\fg \otimes A)$ using Lemma \ref{brack}(3), (5) and (6) and \eqref{xaxa}. For (2) it suffices to show that any product of elements of the set $\{p_i(\chi): \chi \in \mathcal{F}(B), i \in I\}$ is in the $\Z-$span of $\bu_{\Z}^{0}(\fg \otimes A)$. This can be shown by induction on the degree of such a product. For the inductive step we apply Proposition \ref{degp}(4).
\end{proof}

\end{document}